\documentclass[11pt]{amsart}

\usepackage{subfigure}
\usepackage{graphicx}

\newtheorem{thm}{Theorem}[section]
\newtheorem{prop}[thm]{Proposition}
\newtheorem{lem}[thm]{Lemma}
\newtheorem{cor}[thm]{Corollary}

\newtheorem{rem}[thm]{Remark}

\theoremstyle{definition}
\newtheorem{definition}[thm]{Definition}
\newtheorem{example}[thm]{Example}

\newcommand\R{\mathbb R}
\newcommand\weakto\rightharpoonup
\newcommand\Lip{\mathrm{Lip}}

\begin{document}

\title[Material laws from boundary data]
{A well-posed variational approach to the identification and convergent 
approximation of material laws from boundary data} 

\author{S.~Conti${}^{1}$ and M.~Ortiz${}^{2,3}$}

\address
{
    ${}^1$Institut f{\"u}r Angewandte Mathematik, Universit\"at Bonn, 
    Endenicher Allee 60, 53115 Bonn, Germany 
    \\
    ${}^2$Division of Engineering and Applied Science, California Institute 
    of Technology, 1200 E.~California Blvd., Pasadena, CA 91125, USA\\ 
    ${}^3$Institut f{\"u}r Angewandte Mathematik and Hausdorff Center for 
    Mathematics, Universit\"at Bonn, Endenicher Allee 60, 53115 Bonn, 
    Germany
}

\email{ortiz@caltech.edu}

\begin{abstract}
We formulate the problem of material identification as a problem of optimal control in which the deformation of the specimen is the state variable and the unknown material law is the control variable. We assume that the material obeys finite elasticity and that the deformation of the specimen is in static equilibrium with prescribed boundary displacements. We further assume that the attendant total energy of the specimen can be measured, e.~g., with the aid of the work-energy identity. In particular, no full-field measurements, such as DIC, are required. The cost function measures the maximum discrepancy between the total elastic energy corresponding to a trial material law and the measured total elastic energy over a range of prescribed boundary displacements. The question of material identifiability is thus reduced to the question of existence and uniqueness of controls. We propose a specific functional framework, prove existence of optimal controls and show that the question of material identifiability hinges on the separating properties of the boundary data. The proposed framework naturally suggests and supports approximation by {\sl maxout} neural networks, i.~e., neural networks of piecewise affine or polyaffine functions and a maximum, or {\sl join}, activation function. We show that maxout neural networks have the requisite density property in the space of energy densities and result in convergent approximations as the number of neurons increases to infinity. Simple examples are also presented that illustrate the minimax structure of the identification problem.  
\end{abstract}

\maketitle

\section{Introduction}

The availability of big material data sets, made possible by epochal advances 
in experimental and computational science (see, e.~g., \cite{Buljac:2018, 
Schleder:2019, Bernier:2020, Wang:2024}), has given rise to a desire to forge 
a closer nexus between material data and the predictions they enable. Two 
main paradigms have emerged, loosely corresponding to supervised and 
unsupervised methods in the context of machine learning: Model-free 
approaches, in which material set data is combined directly with field 
equations to effect predictions of quantities of interest 
\cite{Kirchdoerfer:2016, Kirchdoerfer:2017, Leygue:2018, conti:2018, 
Stainier:2019, Eggersmann:2019, Conti:2020, Carrara:2020, Prume:2023}; and 
model-based approaches, in which the connection between material data and 
predictions is effected through the intermediate step of identifying a 
material law from the data \cite{Liu:2022, Bhattacharya:2023, Liu:2023, 
Weinberg:2023, Asad:2023, Marino:2023, Ghane:2024, Akerson:2024}.  

The present work focuses on the second paradigm, specifically, on the problem 
of identifying material laws from experimental data. It has been long 
recognized that material identification from empirical data may be regarded 
as an inverse problem (see, e.~g., the pioneering work of Bui \cite{Bui:1993} 
including non-destructive detection, the characterization of internal defects 
and inhomogeneities; the identification of singularities in fracture 
mechanics; and the identification of the physical parameters of materials). 
For the most part, the classical work is concerned mainly with the 
identification of parameters in a given class of models, in contrast to the 
more challenging problem of identifying the functional form of the model 
itself (see, e.~g., \cite{Martins:2018} and references therein). Recent 
advances in experimental techniques, including the ability to perform 
spatially and time-resolved measurements \cite{Sutton:2009,Buljac:2018}, and 
high-throughput experiments \cite{Li:2019, Jin:2022} have provided renewed 
impetus to the subject. In addition, neural networks and machine learning 
have supplied a new and efficient means of representing material laws and 
fitting them by regression to big data sets, causing an extensive 
reevaluation of the field.  

These advances notwithstanding, the choice of cost/loss function used for 
purposes of identification and data regression is often {\sl ad hoc} and does 
not ensure existence of solutions. Additionally, a suitable notion of 
convergence of approximations if often lacking, and no guarantee is offered 
that the approximate solutions obtained, e.~g., through a neural-network 
representation, converge in any suitable sense to the underlying--and 
unknown--material law.  

The present work is concerned with the formulation of a well-posed 
variational framework enabling the identification and convergent 
approximation of material laws from boundary data. By {\sl well-posed} we 
mean that the variational problem has solutions under reasonable assumptions 
on the data, and that the solutions indeed identify the underlying--and 
unknown--material law. By {\sl convergent} approximation we mean that the 
variational framework generates provably convergent approximate solutions, 
provided that the spaces of trial material laws are suitably chosen and are 
dense in the space of energy densities. Other properties of the variational 
property, such as ease of implementation and computational efficiency, are 
also desirable, but are not addressed in the present work (interested readers 
are referred to \cite{Conti:2024}). 

We specifically address these questions within the framework of finite 
elasticity, where the objective is to identify the strain-energy density of 
the material. We recall that finite elasticity challenges material 
identification because it allows for local and global instabilities, such as 
internal buckling and cavitation \cite{Ball:1982, Ball:1984}. For 
definiteness, we assume that the experimental setup enables displacement 
control over the entire boundary of the specimen and results in a measurement 
of the total energy of specimen, e.~g., by integrating the work-energy 
identity along convenient loading paths. It bears emphasis that in this 
scenario no full-field measurements, such as obtained from digital image 
correlation (DIC), are assumed or required. Instead, the goal is to achieve 
material identification solely from boundary data. We also emphasize that the 
boundary data is not restricted to--and is not required to include--affine 
boundary displacements.  

We note that the assumption of displacement control is necessitated in 
practice by the requirement of existence of equilibrium solutions in finite 
elasticity \cite{Ball:1976}. The situation is far less rigid if only small 
strains are allowed for or anticipated. In this case, an appeal to standard 
duality arguments and to devices such as the Dirichlet-to-Neumann map enable 
consideration of mixed traction/displacement boundary conditions. By 
contrast, extensions to mixed boundary conditions in the presence of large 
deformations are not covered by present theory and remain open 
\cite{Ball:2002}. 

We show that these questions can be addressed systematically, both from a 
theoretical as well as a computational perspective, by formulating the 
problem of material identification as a problem of optimal control 
\cite{Belloni:1993, Bucur:2005} in which the deformation of the specimen is 
the state variable and the unknown material law is the control variable. We 
assume that the deformation is in static equilibrium with prescribed 
displacements imparted by a loading device. The cost function then measures 
the discrepancy between the elastic energy corresponding to a trial material 
law and the measured elastic energy for a sample of loading conditions. The 
question of identifiability of the material law is thus reduced to the 
question of existence and uniqueness of controls. 

We propose a specific functional framework combining: a bounded set of 
Lipschitz-continuous displacements, setting forth {\sl a priori} bounds on 
the deformation gradients that can be attained by the specimen; and a bounded 
set of Lipschitz-continuous energy densities, setting forth {\sl a priori} 
bounds on the attainable stresses. We endow these spaces with weak and strong 
uniform convergence topologies. We then postulate a specific class of {\sl 
minimax} cost functions and prove that they ensure existence of solutions in 
the form of optimal controls. We additionally show that the question of 
identifiability of the material law hinges on the separating properties of 
the boundary data.  

We additionally show that the proposed framework, and specifically the 
lattice structure of the space of controls, naturally supports approximation 
by {\sl maxout} neural networks \cite{Anil:2019, Goodfellow:2013}, i.~e., 
neural networks with affine or polyaffine neurons and a maximum, or join, 
activation function. We show that {\sl maxout} neural networks have the 
requisite density property in the space of controls, or the universal 
approximation property in machine-language {\sl newspeak}, and result in 
convergent approximations as the number of neurons increases to infinity. 
Simple examples are also presented that illustrate the minimax structure of 
the identification problem.  

\section{Problem formulation}

The general optimal control problem takes the form \cite{Belloni:1993, 
Bucur:2005} 
\begin{equation}\label{eq:OC:P1}
    \inf \{ J(u,y) : u \in U, \; y\in {\rm argmin} \, G(u,\cdot) \}
\end{equation}
where $y\in Y$ is the {\sl state variable} and $Y$ is the {\sl state space}, 
$u\in U$ is the {\sl control variable} and $U$ is the {\sl control set}, $J : 
U\times Y \to \overline{\mathbb{R}}$ is the {\sl cost functional} and $G : 
U\times Y \to \overline{\mathbb{R}}$ is the {\sl state functional}. 

In applying the general optimal control framework to material identification, 
we make the following choices and assumptions: 
\begin{itemize}
\item[A1)] State space. We restrict to deformations $y$ of the solid that 
belong to a closed and bounded subset $Y$ of $W^{1,\infty}(\Omega; 
\mathbb{R}^n)$, with $\Omega \subset \mathbb{R}^n$ open, bounded, connected, 
Lipschitz. Specifically, we let $Y$ be the closed and bounded subset of 
$W^{1,\infty}(\Omega;\mathbb{R}^n)$ defined by 
\begin{equation}\label{eqboundsY}
\begin{split}
    \left|\int_\Omega y(x)\, dx\right| \leq C_0 < +\infty,
    \quad &
    \mathop{\mathrm{ess\,sup}}\limits_{x\in\Omega}  |Dy(x)| 
    \leq C_1 < +\infty ,
    \\ &
    \det Dy(x)\ge 0 \text{ a.e.}
\end{split}
\end{equation}
We shall assume that $C_0>C_1|\Omega|\mathrm{diam \,}\Omega$. In addition, we 
denote by 
\begin{equation}\label{k8cJfL}
    K 
    = 
    \{ 
        \xi \in \mathbb{R}^{n\times n} \, : \, 
        |\xi| \leq C_1, \, \det \xi\ge0
    \}
\end{equation}
the compact set of local deformation gradients covered by $Y$. We endow $Y$ 
with the strong and weak-$*$ topologies inherited from $W^{1,\infty}$, and 
denote the corresponding convergences by $y_h \to y$ and $y_h \rightharpoonup 
y$, respectively. 

\item[A2)] Control space. The control space $U$ is the space of energy 
densities $u$ which are polyconvex and uniformly Lipschitz. In order to make 
this precise, let us denote by ${\rm Minors}(\xi)$ the set of minors of 
$\xi$, specifically, ${\rm Minors}(\xi)=(\xi,\det\xi)$ in two dimensions and 
$(\xi,\mathrm{cof\,}\xi,\det\xi)$ in three dimensions. We denote by $\tau(n)$ 
the total dimension of this vector (so that $\tau(2)=5$ and $\tau(3)=19$). We 
then require that there is a convex function $g:\R^{\tau(n)} \to [0,\infty)$ 
such that 
\begin{equation}\label{eqdefuqc}
    u(\xi)=g({\rm Minors}(\xi)) \text{ for all } \xi\in K
\end{equation}
and that $g$ is $\ell$-Lipschitz, meaning that
\begin{equation}
    |g(m)-g(m')|\le \ell|m-m'|
\text{ for all }m,m'\in \R^{\tau(n)}.
\end{equation}
We endow $U$ with the topology of uniform convergence.

\item[A3)] State functional: The deformations $y$ of the solid are assumed to 
minimize the potential energy 
\begin{equation} \label{8jtnW5}
    G(u,y;g) 
    =
    \left\{
        \begin{array}{ll}
            \int_\Omega
                u(Dy(x))
            \, dx ,
            & \text{if} \; T y = g \; \text{on} \; \partial\Omega , \\
            +\infty, & \text{otherwise} ,
        \end{array}
    \right.
\end{equation}
over all $y\in Y$. Here, $T$ is the trace operator, $g \in M \subset T Y 
\subset$ $W^{1,\infty}(\partial\Omega; \mathbb{R}^n)$ are prescribed 
displacements. The space ${M}$ describes all possible prescribed boundary 
displacements that can be imparted to the specimen by the loading device. 
We further denote by 
\begin{equation} \label{hmrsP3}
    E(u;g) = \inf G(u,\,\cdot;g) ,
\end{equation}
the infimum energy of the system characterized by a trial energy density 
$u$ under prescribed displacements $g \in M$. 
 
\item[A4)] Cost function. 
For given $E_0 : M \to \mathbb{R}$ and fixed $g \in M$, we let 
\begin{equation} \label{VgsvM1}
    I(u;g)
    =
   \left\{
        \begin{array}{ll}
            E(u;g) - E_0(g)  , 
            & \text{if} \; E_0(g) \leq E(u;g)    , \\
            +\infty , & \text{otherwise} .
        \end{array}
    \right.
\end{equation}
The total cost is then defined as 
\begin{equation} \label{ptMeIB}
    J(u) = \sup \, \{ I(u;g) \, : \, g \in M \} .
\end{equation}
This cost function supplies a measure of energy discrepancy between trial 
and measured values over the entire range $M$ of applied displacements. 
\end{itemize}

The cost function can equivalently be defined by introducing the constraint 
set 
\begin{equation} \label{YZta6C}
    C = \{ u \in U \, : \, E(u;g) \geq E_0(g), 
    \;\; {\text{for every}} \; g \in M \} ,
\end{equation}
which restricts trial energy densities to those whose total energy 
$E(u;\cdot)$ majorizes $E_0(\cdot)$. We then define the cost function as 
\begin{equation} \label{caTe55}
    J(u) = \sup_{g\in M} \big( E(u;g) - E_0(g) \big) + I_C(u) ,
\end{equation}
where $I_C$ is the indicator function of $C$, to be minimized with respect to 
$u$ in $U$. Thus, if $u\in C$, then $I(u,g)=E(u;g)-E_0(g)$ for every $g \in M$,
{and the expression in
\eqref{ptMeIB}  equals the one in  \eqref{caTe55}.} If $u\not\in C$, then there is $g$ such that
$E(u;g)<E_0(g)$, so that $I(u;g)=+\infty$, and both {the expression in 
\eqref{ptMeIB} and the one in \eqref{caTe55} are $+\infty$.} 

The preceding choices of state and control spaces, and their chosen 
topologies, ensure stability with respect to perturbations, in the sense that 
small errors in material behavior result in small errors in the state of 
deformation of the specimen. The analysis that follows also shows that the 
choice (\ref{ptMeIB}) of cost function is natural in the sense that the 
corresponding optimal control problem admits solutions that uniquely identify 
the underlying 'ground truth' material behavior from boundary data. 

Furthermore, the cost function has a natural {\sl monotonicity property} with 
respect to boundary data. Thus, suppose that two experimental programs are 
characterized by prescribed displacements $M'$ and $M''$. Then, it follows 
immediately from (\ref{ptMeIB}) that, if $M' \subset M''$, the corresponding 
cost functions are ordered as $J' \leq J''$. It thus follows that the 
experimental program $M''$ is more {\sl constraining}, i.~e., it assigns a 
higher cost to discrepancies in material behavior, than the program $M'$. 

\section{Identifiability from boundary data}

The optimal control problem set forth in the foregoing is
\begin{equation}\label{SC3dBH}
    \inf_{u\in U} J(u) ,
\end{equation}
i.~e., we seek to find the energy densities $u$ that minimize the discrepancy 
between predicted and measured total energy over the entire range of applied 
displacements. In addition, if $E_0(g) = \inf G(w,\cdot;g)$ for some unknown 
energy density $w \in U$, we wish to ascertain conditions under which $w$ is 
recovered as the unique solution of optimal control problem. 

Unless otherwise stated, through this section we assume that assumptions 
(A1--A4) are in force. 

\subsection{Prolegomena and notational conventions}

We set forth notation and summarize relevant background for completeness and 
ease of reference. 

\begin{definition}[$\Gamma$-convergence]\label{z1t0CI}
A sequence $F_h : X \to 
\overline{\mathbb{R}}{=\mathbb{R}\cup\{-\infty,+\infty\}}$ is said to {\sl 
$\Gamma$-converge} to a functional $F: X \to \overline{\mathbb{R}}$ if: 
\begin{itemize}
\item[i)] For every $x_h\to x$, $F(x) \leq \liminf_{h\to\infty} F_h(x_h)$.
\item[ii)] For every $x \in X$, there is $x_h\to x$ such that $F(x) \geq 
    \limsup_{h\to\infty} F_h(x_h)$. 
\end{itemize}
\end{definition}

The following theorems summarize the main properties of $\Gamma$-convergence 
(see, e.~g., \cite{Maso:1993}). We assume throughout that $X$ is a separable 
metric space. 

\begin{thm}[$\Gamma$-convergence]\label{HV42Sq} ${}$
\begin{itemize}
\item[i)] Every $\Gamma$-limit is lower-semicontinuous on $X$. 
\item[ii)] If $F_h$ is equicoercive on $X$ and $\Gamma$-converges to $F$, 
    then $F$ is coercive and admits a minimum in $X$. If, in addition, $F$ 
    is not identically $+\infty$ and $x_h\in{\sl argmin}\, F_h$, then there 
    exists a subsequence that converges to an element of ${\sl
    argmin} \, F$. 
\item[iii)] From every sequence $F_h$ of functionals on $X$ it is possible 
    to extract a subsequence $\Gamma$-converging to a functional $F$ on 
    $X$. 
\end{itemize}
\end{thm}

{The {\sl lower semicontinuous envelope} ${\rm sc}^- F: X \to \overline { 
\mathbb{R}}$  of a functional $F$ is defined as the supremum of all lower 
semicontinuous functionals majorized by $F$, i.~e., 
\begin{equation}
    {\rm sc}^- F(x) 
    = 
    \sup 
    \left\{ 
        G(x) \,:\, G \; \text{l.s.c.}, \ G\le F 
    \right\} ,
\end{equation}
where we write $G\le F$ if $G(x) \leq F(x)$ for all $x\in X$. It is easy to 
see that ${\rm sc}^-F$ is lower semicontinuous and 
\begin{equation}\label{defRel}
    {\rm sc}^- F(x)
    =
    \inf \left\{
    \liminf_{h\to\infty} F(x_h) :
    x_h \in X\,,\ x_h\to x\right\}\,.
\end{equation}
For any coercive functional $F$, the following holds \cite[p.30]{Maso:1993}:
\begin{enumerate}
\item ${\rm sc}^- F$ is coercive and  lower semicontinuous. 
\item ${\rm sc}^- F$ admits at least a minimum point. 
\item $\min_{x \in X} {\rm sc}^- F(x) = \inf_{x \in X} F(x)$. 
\item If $x$ is the limit of a minimizing sequence for $F$, then $x$ is a 
    minimum point for ${\rm sc}^- F$. 
\item If $x$ is a  minimum point for ${\rm sc}^- F$, then $x$ is the limit 
    of a minimizing sequence of $F$. \\ 
\end{enumerate}
}

We denote 
\begin{equation}
    (f \vee g)(x) = \max(f(x), g(x)),
    \quad
    (f \wedge g)(x) = \min(f(x), g(x)).
\end{equation}
For an index function $I$ and functions $\{f_i\}_{i\in I}$, we also write
\begin{equation}
    \bigvee_{i\in I} f_i(x)
    =
    \max\{f_i(x)\}_{i\in I} ,
    \quad
    \bigwedge_{i\in I} f_i(x)
    =
    \min\{f_i(x)\}_{i\in I} .
\end{equation}
A set of functions $L$ is a {\sl lattice} if for every $f$, $g \in L$ we have 
$f \wedge g \in L$ and $f \vee g \in L$.
 
\subsection{Properties of the state functional}

We begin by establishing basic properties of state functional $G$. In 
particular, we show that the assumed functional setting ensures the existence 
of minimizing states in $Y$ for all energy densities in $U$. In addition, we 
show that the functional setting is natural with respect to approximation, in 
the sense that convergence of  sequences of energy densities $(u_h)$ in $U$ 
ensures convergence of corresponding sequences of minimizers $(y_h)$ in $Y$. 
This property is of critical importance since, in practice, the energy 
density of a material is unlikely to be identified exactly. Therefore, it is 
natural to require that predicted states be increasingly accurate as the 
energy density itself is identified with increasing accuracy. 

We first observe that polyconvexity of $u$, in the sense of~\eqref{eqdefuqc}, 
implies lower semicontinuity of the corresponding integral functional. This 
can be embedded in the standard formulation making $u$ extended-valued, the 
treatment in this case is however simpler since the set where $u$ is finite 
is compact and polyconvex. Therefore we sketch the standard argument. 

\begin{prop}\label{lemmaqclsc}
Let $u\in U$. Then $y\mapsto \int_\Omega u(Dy)\,dx$ is  lower semicontinuous 
with respect to weak-* $W^{1,\infty}(\Omega;\R^n)$ convergence {in $Y$, and 
therefore with respect to uniform convergence in $Y$.} 
\end{prop}

\begin{proof}
Let $y_h\in Y$, $y_h\rightharpoonup y$ in the weak-$*$  $W^{1,\infty}$ 
topology. This in particular implies uniform convergence and weak convergence 
of the minors, in the sense that 
\begin{equation}
{\rm Minors}(y_h)\rightharpoonup {\rm Minors}(y)
\end{equation}
in the weak-$*$ $W^{1,\infty}$ topology. From $|Dy_h|\le C_1$ almost 
everywhere and $Dy_h\weakto Dy$ we deduce $|Dy|\le C_1$ almost everywhere. 
Analogously, from $\det Dy_h\ge0$ and $\det Dy_h\weakto \det Dy$ we deduce 
$\det Dy\ge0$ almost everywhere, so that $Dy\in K$ almost everywhere. From 
uniform convergence one analogously obtains that the first condition in 
\eqref{eqboundsY} holds for $y$, so that $y\in Y$. 

In order to prove lower semicontinuity of the energy, let $g: \R^{\tau(n)} 
\to [0,\infty)$ be a convex function as in~\eqref{eqdefuqc}. Then
\begin{equation}
    \int_\Omega u(Dy_h)dx
=\int_\Omega g({\rm Minors}(Dy_h))dx,
\end{equation}
and by convexity one obtains
\begin{equation}
    \int_\Omega u(Dy)dx
    =\int_\Omega g({\rm Minors}(Dy))dx
    \le \liminf_{h\to\infty}\int_\Omega g({\rm Minors}(Dy_h))dx.
\end{equation}
This proves lower semicontinuity.

{Since every sequence in $Y$ is bounded in $W^{1,\infty}(\Omega;\R^n)$, every 
sequence in $Y$ which converges uniformly also converges with respect to the  
weak-$*$ topology of $W^{1,\infty}(\Omega;\R^n)$.} 
\end{proof}

The existence of minimizing deformations now follows from the following 
standard argument due to Tonelli \cite{Tonelli:1921}. 

\begin{cor} \label{ZHSbDO}
For every $g\in M$ and $u\in U$ the functional $G(u,\cdot;g)$ is not 
everywhere $+\infty$ and has a minimizer in $Y$. 
\end{cor}

\begin{proof}
Since $M \subset T Y$, there is $y \in Y$ such that $T y = g$. For this $y$, 
\begin{equation}
    G(u,y;g)
    =
    \int_\Omega u(Dy(x)) \, dx
    \leq
    \int_\Omega(u(0)+C_1\ell)\, dx
    <
    +\infty ,
\end{equation}
and the first claim follows.

Let $y_h$ be a minimizing sequence. As $Y$ is weakly compact, there is a 
subsequence which converges weakly to some $y\in Y$. By 
Prop.~\ref{lemmaqclsc} this implies 
\begin{equation}
    G(u,y;g) 
    \le
    \liminf_{h\to\infty} G(u,y_h;g) ,
\end{equation}
and $y$ is a minimizer.
\end{proof}

Next we turn to the question of approximating energy densities and attendant 
convergence of minimizers thereof.

\begin{prop}\label{lemmaUclosed}
${}$
\begin{itemize}
\item[i)] The set $U$ is closed, in the sense that if $u_j\in U$, and 
    $u_j\to u$ uniformly, then $u\in U$. 
\item[ii)]The set $U$ is compact {modulo constants}, in the sense that for 
    any sequence $u_j\in U$ there are a subsequence $u_{j_k}$ {and
    constants $c_j\in\R$ such that $u_{j_k}-c_{j_k}$} converges uniformly 
    to some $u\in U$. 
\end{itemize}
\end{prop}
\begin{proof}
 i) Let $u_j\in U$, $u_j\to u$ uniformly for some $u:K\to\R$.
Let $g_j:\R^{\tau(n)}\to\R$ be such that $u_j=g_j\circ{\rm Minors}$. The 
functions $g_j$ are uniformly $\ell$-Lipschitz, and $g_j(0)=u_j(0)\to u(0)$. 
By the Ascoli-Arzel\'a theorem for every $R>0$ there is a subsequence that 
converges uniformly on $[-R,R]^{\tau(n)}$, by a diagonalization argument we 
obtain a unique subsequence that converges locally uniformly to some 
$g:\R^{\tau(n)}\to\R$. {This implies $u=g\circ{\rm Minors}$. Further, locally 
uniform convergence implies that $g$ is also $\ell$-Lipschitz.} Therefore 
$u\in U$. 

ii) {We choose $c_j:=u_j(0)$. As the} functions $u_j-c_j$ are uniformly 
Lipschitz on the compact set $K$, {they are bounded, by the} Ascoli-Arzel\`a 
theorem they have a uniformly converging subsequence. The assertion then 
follows {by  i)}. 
\end{proof}

Convergence of minimizers is now ensured by the following 
$\Gamma$-convergence property of the state functionals.

\begin{prop}\label{SH8qXX}
For every $g\in {M}$ and every sequence $u_h\to u$ in $U$, 
\begin{equation} \label{cqC8FR}
    \Gamma\text{-}\lim_{h\to\infty} G(u_h,\cdot;g) = G(u,\cdot;g),
\end{equation}
weak-$*$ in $Y$.
\end{prop}

\begin{proof}
We refer to Def.~\ref{z1t0CI} for the definition of $\Gamma$-convergence. i) 
Suppose that $u_h \to u$ in $U$ (i.e., uniformly) and $y_h \rightharpoonup y$ 
in $Y$ (i.e., weak-$*$ in $W^{1,\infty}$). By definition of $Y$, $Dy_h\in K´$ 
almost everywhere for all $h$. As $u_h\to u$ uniformly on $K$, we have 
\begin{equation}\label{eqcompconv} \lim_{h\to\infty} 
    \sup_{\xi\in K} |u_h(\xi)-u(\xi)|=0.
\end{equation}
Therefore, as $\Omega$ is bounded and $Ty_h=g$ for all $h$,
\begin{equation}\label{BkXDr8a}
\begin{split}
    \liminf_{h\to\infty} G(u_h,y_h;g)
    =&
    \liminf_{h\to\infty} \int_\Omega u_h(Dy_h(x)) \, dx
    = 
    \liminf_{h\to\infty} \int_\Omega u(Dy_h(x)) \, dx.
\end{split}
\end{equation}
By Prop.~\ref{lemmaqclsc},
\begin{equation}\label{BkXDr8}
\begin{split}
    \liminf_{h\to\infty} \int_\Omega u(Dy_h(x)) \, dx
    \geq
    \int_\Omega u(Dy(x)) \, dx
    =
    G(u,y;g) .
\end{split}
\end{equation}
Since $y_h\weakto y$ implies in particular uniform convergence, from
$Ty_h=g$ we obtain $Ty=g$, so that
\begin{equation}\label{BkXDr8aa}
\begin{split}
    \int_\Omega u(Dy(x)) \, dx
    =
    G(u,y;g) ,
\end{split}
\end{equation}
which concludes the proof.

ii) Suppose that $u_h \to u$ in $U$ and $y \in Y$ with $Ty=g$. By
\eqref{eqcompconv},
\begin{equation}\label{BkXDr8b}
\begin{split}
    \limsup_{h\to\infty} G(u_h,y;g)
    & =
    \limsup_{h\to\infty} \int_\Omega u_h(Dy(x)) \, dx
    \\ & = 
    \int_\Omega u(Dy(x)) \, dx = G(u,y;g) ,
\end{split}
\end{equation}
as required. 
\end{proof}

\begin{cor} \label{7nwWkA}
Suppose that $g\in {M}$ and $u_h\to u$ in $U$. Let $(y_h)$ be a corresponding 
low-energy sequence in $Y$, i.~e., 
\begin{equation} \label{mggFm3}
    G(u_h,y_h;g) \leq \inf G(u_h,\cdot;g) + \epsilon_h ,
\end{equation}
for some sequence $\epsilon_h \downarrow 0$. Then, there are a subsequence 
and $y \in {\rm argmin} \, G(u,\cdot;g)$ such that $y_h \rightharpoonup y$ in 
$Y$. 
\end{cor}

\begin{proof}
Since $Y$ is weakly compact, there is a subsequence of $(y_h)$, not renamed, 
which converges weakly to some $y\in Y$. By the $\Gamma$-convergence of 
$G(u_h,\cdot;g)$, Prop.~\ref{SH8qXX}, and (\ref{mggFm3}), we have 
\begin{equation}
    \inf G(u,\cdot; g)
    =
    \lim_{h\to\infty} \inf G(u_h,\cdot;g)
    \geq
    \liminf_{h\to\infty} G(u_h,y_h;g) 
    \geq 
    G(u,y;g) ,
\end{equation} 
and $y \in {\rm argmin} \, G(u,\cdot;g)$, as surmised.
\end{proof}

Convergence of minimizers can in fact be strong under certain circumstances. 

\begin{prop} [Strong convergence]
Under the conditions of Prop.~\ref{SH8qXX} and Cor.~\ref{7nwWkA}, 
if additionally, for $y \in {\rm argmin} \, G(u,\cdot;g)$ and all $y' \in Y$, 
\begin{equation}\label{eqassGuniq}
    G(u,y';g) - G(u,y;g) 
    \geq
    \varphi( \| y' - y \|_{{\rm Lip}(K)} ) ,
\end{equation}
for some continuous, strictly increasing function over $[0,+\infty)$ with 
$\varphi(0) = 0$, then convergence of minimizers is actually strong. 
\end{prop}
We remark that in the special case $g(x) = \xi x$, $\xi \in 
\mathbb{R}^{n\times n}$, by polyconvexity of $u$ \eqref{eqassGuniq} reduces
to 
\begin{equation}
    G(u,y';g) - u(\xi) \, |\Omega| 
    \geq
    \varphi( \| y' - \xi x \|_{{\rm Lip}(K)} ) ,
\end{equation}
which requires $u$ to be strictly quasiconvex at $\xi$.
\begin{proof}
By $\Gamma$-convergence, we have 
\begin{equation}
    \lim_{h\to\infty} G(u_h,y_h;g) = G(u,y;g) .
\end{equation} 
Write
\begin{equation}
\begin{split}
    G(u_h,y_h;g) - G(u,y;g)
    & = 
    G(u_h,y_h;g) - G(u,y_h;g) 
    \\ & +
    G(u,y_h;g) - G(u,y;g) .
\end{split}
\end{equation} 
By uniform convergence, 
\begin{equation}
    \lim_{h\to\infty} 
    \Big(
        G(u_h,y_h;g) - G(u,y_h;g) 
    \Big) 
    =
    0 .
\end{equation} 
Hence, 
\begin{equation}
    \lim_{h\to\infty} 
    \Big(
        G(u,y_h;g) - G(u,y;g) 
    \Big) 
    =
    0 .
\end{equation} 
 Then, by~\eqref{eqassGuniq},
\begin{equation}
    \lim_{h\to\infty} 
    \| y_h - y \|_{{\rm Lip}(K)}
    =
    0 ,
\end{equation} 
and the convergence of the minimizers is indeed strong. 
\end{proof}

\subsection{Existence of optimal controls}

We recall that $y_h \to y$ and $y_h \rightharpoonup y$ denote strong and weak 
convergence in $Y$, respectively, and that we are using uniform convergence 
on $U$. 

The following propositions set forth basic properties of the cost functional 
$J$. 

\begin{prop}[Constraint set] \label{uH5p92} 
The set $C$, eq.~(\ref{YZta6C}), is strongly closed and convex. 
\end{prop}

\begin{proof}
For any fixed $y$, $u\mapsto G(u,y;g)$ is linear and continuous (with respect 
to uniform convergence of $u$). In particular, it is concave and upper 
semicontinuous. Both properties are preserved by taking the $\inf$ (just like 
convexity and lower semicontinuity with the $\sup$), therefore $E(u;g)=\inf_y 
G(u,y;g)$ is concave and upper semicontinuous. The superset of a concave, 
upper-semicontinuous function is convex and closed. Actually, {one can show 
that $E$} is continuous, see lemma~\ref{lemmaUC0}. 

An alternative more explicit argument is as follows. Suppose that the 
sequence $(u_h)$ is contained in $C$ and converges strongly to $u \in U$. 
Then, by prop.~\ref{SH8qXX}, the properties of $\Gamma$-convergence and the 
compactness of $U$, prop.~\ref{lemmaUclosed}{ii}, for a.~e.~$g \in M$ it 
follows that 
\begin{equation}
    E_0(g)
    \leq
    \lim_{h\to\infty} E(u_h;g)
    =
    \lim_{h\to\infty} \inf G(u_h,\cdot;g)
    =
    \inf G(u,\cdot;g)
    =
    E(u;g) ,
\end{equation}
hence $u \in C$. 

For fixed $g \in M$, let $u_1$, $u_2 \in U$, $\lambda_1\geq 0$, $\lambda_2 
\geq 0$, $\lambda_1+\lambda_2=1$. By Corollary~\ref{ZHSbDO}, there is $y^*\in 
Y$ such that $T y^* = g$ and 
\begin{equation} 
    E(\lambda_1u_1+\lambda_2u_2;g) 
    = 
    G(\lambda_1u_1+\lambda_2u_2,y^*;g) 
\end{equation}
in (\ref{hmrsP3}). From (\ref{8jtnW5}) we then have
\begin{equation}
\begin{split}
    G(\lambda_1u_1+\lambda_2u_2,y^*;g)
    & =
    \int_\Omega
        \Big(
            \lambda_1 u_1(Dy_*(x))
            +
            \lambda_2 u_2(Dy_*(x))
        \Big)
    \, dx 
    \\ & \geq
    \lambda_1 E(u_1;g) + \lambda_2 E(u_2;g) ,
\end{split}
\end{equation}
which shows that $E(u;g)$ is a concave function of $u$. {If $u_1$, $u_2\in C$ 
then 
    $E(u_1;g)\ge E_0(g)$ and $E(u_2;g)\ge E_0(g)$, therefore
\begin{equation}
\begin{split}
    E(\lambda_1u_1+\lambda_2u_2;g)=
    G(\lambda_1u_1+\lambda_2u_2,y^*;g)\ge E_0(g),
\end{split}
\end{equation}
and $\lambda_1u_1+\lambda_2u_2\in C$.}
\end{proof}

\begin{lem}\label{lemmaUC0} For every $g\in M$ the map $u\mapsto E(u;g)$ is 
continuous, with respect to uniform convergence of $u$.
\end{lem} 
\begin{proof}
Assume that $u$, $u'\in U$, with $u\le u'+\delta$ pointwise for some 
{$\delta\in\mathbb{R}$}. For every $y$, we have
\begin{equation}
    E(u;g)\le G(u,y;g)\le G(u'+\delta,y;g)=G(u',y;g)+\delta|\Omega|.
\end{equation}
Taking the infimum over all $y$ leads to
\begin{equation}\label{esEuu1}
    E(u;g)\le E(u';g)+\delta|\Omega|.
\end{equation}
The other inequality follows by swapping $u$ and $u'$. This proves that $E$ 
is Lipschitz continuous, with Lipschitz constant equal to the measure of 
$\Omega$. 
\end{proof}

\begin{thm}[Lower semicontinuity] \label{UCPPzW}
The cost function $J$, eq.~(\ref{caTe55}), is sequentially weak-$*$ lower 
semicontinuous on $U$. 
\end{thm}

\begin{proof}
Since $E(u;g)$ is continuous with respect to uniform convergence of $u$, the 
supremum $\sup_{g\in M} (E(u;g)-E_0(g))$ is lower semicontinuous with respect 
to the same topology. As the weak-* topology of $W^{1,\infty}$ is stronger 
than uniform convergence, this concludes the proof.
\end{proof}

\begin{cor}[Existence] \label{NQa6VY}
$J$ has a minimizer in $U$.
\end{cor}

\begin{proof}
Existence follows from theorem~\ref{UCPPzW}, compactness of $U$ proven in 
Prop.~\ref{lemmaUclosed}, and Tonelli's theorem \cite{Tonelli:1921}. 
\end{proof}

\subsection{Uniqueness and identifiability}

Next we turn to the question of uniqueness of the optimal controls. 

\begin{definition} [Separating property] \label{Lawj4j}
We say that the set $M$ of prescribed boundary displacements is {\sl 
separating} if, for every $u$, $v \in U$ such that $u \neq v$, there is $g\in 
M$ such that $E(u;g) \neq E(v;g)$. 
\end{definition}

\begin{example}[Affine boundary conditions]
Let 
\begin{equation}
    M = \{ \xi x \,:\, \xi \in K \} . 
\end{equation}
Then, by {polyconvexity} of $u$,
\begin{equation}
    E(u;\xi x) = |\Omega| \, u(\xi)
\end{equation}
{for any $u\in U$.} Suppose $u$, $v \in U$ and $u \neq v$. Then, there is 
$\xi \in K$ s.~t., $u(\xi) \neq v(\xi)$ and, therefore, $E(u;\xi x) \neq 
E(v;\xi x)$. Hence, $M$ is separating. By the same token, 
\begin{equation}
    M = \{ \xi x \,:\, \xi \in K \cap \mathbb{Q}^{n\times n} \} ,
\end{equation}
with $\mathbb{Q}$ the field of rational numbers, defines a {\sl countable} 
separating set of prescribed boundary displacements. \hfill$\square$
\end{example}

\begin{prop}[Uniqueness]\label{propunique}
Suppose that there is $w \in U$ such that $E_0(g) = E(w;g) = \inf  
G(w,\,\cdot; g)$. Suppose that $M$ is separating. Then, $J$ has a unique 
minimizer in $U$. 
\end{prop}

\begin{proof}
{We first observe that $\inf J=J(w)=0$.} Suppose that $J(u) = J(v) = \inf J = 
0$ and $v\neq u$. Then, there is $g \in M$ such that $E(u;g) \neq E(v;g)$. 
Suppose that $E(v;g) < E(u;g)$. {From $J(u)=0$ we obtain $E(u;g)=E(w;g)$, 
therefore} $E(v;g) < E(w;g)$ and $I(v;g) = +\infty$, in contradiction with 
the assumption that $v$ is a minimizer. Hence, the minimizer is unique.  
\end{proof}

\begin{cor}[Identifiability] \label{D94WVn}
Suppose that there is $w \in U$ such that $E_0(g) = E(w;g) = \inf  
G(w,\,\cdot; g)$. Suppose that $M$ is separating and let $u_*$ be the unique 
minimizer of $J$. Then, $u_* = w$. 
\end{cor}

\begin{proof}
From (\ref{VgsvM1}) and (\ref{ptMeIB}) it follows that $J \geq 0$. In 
addition $J(w) = 0$ and, hence, $w$ is a minimizer. By uniqueness, $u_* = w$. 
\end{proof}

It thus finally follows that the problem (\ref{SC3dBH}) indeed identifies 
variationally the energy density of the solid being tested, provided that the 
experimental program is thorough enough that the prescribed boundary 
displacements be separating.  

\begin{rem}{\rm
If the underlying energy density $w$ is not quasiconvex, then one can only 
hope to identify the relaxed energy $Qw$, given by the quasi-convexification 
of $w$. A rigorous statement is however difficult, due to the presence of the 
determinant constraint in the definition of $U$. Indeed, relaxation theory 
with determinant constraints is only partially known. We remark that in the 
scalar case, where quasiconvexity {and polyconvexity coincide with} 
convexity, the situation is much simpler.} 
\end{rem}

\subsection{Illustrative examples}

We illustrate the essential tradeoffs at play by means of the following 
elementary examples. 

\subsubsection{One bar} \label{UMGK5s} 
Consider a bar of length $L$ occupying the domain $\Omega=[0,L]$, subject to 
boundary conditions 
\begin{equation}
    y(0) = 0, \quad y(L) = g .
\end{equation}
For every $y\in Y$, we have 
\begin{equation}
    |y(L)| 
    = 
    \Big| \int_0^L Dy(x) \, dx \Big|
    \leq
    \int_0^L | Dy(x) | \, dx
    \leq
    C_1 L .
\end{equation}
As in one dimension the condition on the determinant reduces to $y'\ge0$, we 
also obtain $y(L)\ge0$. For $\xi\in [0,C_1]$, let $y(x) = \xi x$. Assume 
$C_0\ge C_1L$. Then, $Dy = \xi$, $y\in Y$ and $y(L) =\xi L$. Hence, $TY = [0, 
C_1L]$ and $K= [0, C_1]$. Take $M \subset TY$. Then, for every $g \in M$, we 
have 
\begin{equation}
    E(u;g) = L \, u(g/L) ,
    \quad
    E_0(g) = L \, w(g/L) .
\end{equation}
From (\ref{VgsvM1}) 
\begin{equation}
    I(u;g)
    =
    \left\{
        \begin{array}{ll}
            {L \, u(g/L) - L \, w(g/L)}  ,
            & \text{if} \quad L \, w(g/L) - L \, u(g/L) \leq 0 , \\
            +\infty , & \text{otherwise} .
        \end{array}
    \right.
\end{equation}
The cost function $J(u)$ then follows from (\ref{ptMeIB}) as 
\begin{equation}
    J(u)
    =
    \left\{
        \begin{array}{ll}
            \operatorname{sup} \{ u(\xi) - w(\xi)  \,:\, \xi L \in M\} ,
            & \text{if} \; u \geq w \; \text{on} \; M/L , \\
            +\infty , & \text{otherwise} .
        \end{array}
    \right.
\end{equation}
The minimizers satisfy 
\begin{equation}
    u(\xi) = w(\xi) , \quad \xi \in M/L ,
\end{equation}
and are arbitrary functions with $\operatorname{Lip}(u)\leq \ell$ elsewhere 
in $K$. Thus $w$ is identified uniquely in $M/L$ {and, by continuity, on its 
closure}---albeit not elsewhere in $K$---by the minimizers of $J$. 
Furthermore, the minimizer $u$ is unique if $\overline M = K L$, in which 
case $M$ is separating. 

\subsubsection{Two bars in parallel} \label{9AqoKC} 

The case of two bars in parallel provides an example in which uniform 
deformations generated by affine boundary conditions are not available. 
Suppose that the lengths of the bars are $L_1$, $L_2$, the cross-sectional 
areas $A_1$ and $A_2$, respectively. The bars are pinned on one end, jointed 
at the other end, and deform under the action of a prescribed displacement 
$\delta$ applied to the jointed end. For a trial energy density per unit 
volume $u \in U$, the corresponding minimum energy is 
\begin{equation} \label{z8ezZq}
\begin{split}
    &
    E(u;\delta)
    = 
    \inf
    \{
        A_1 \int_0^{L_1} u(Dy_1(x)) \, dx 
        + 
        A_2 \int_0^{L_2} u(Dy_2(x)) \, dx
        \,:\, \\ & \qquad\qquad
        y_1(0) = y_2(0) = 0 ;\;
        y_1(L_1) = L_1+\delta,\; y_2(L_2) = L_2+\delta
    \}
    = \\ &\qquad\qquad
    A_1 L_1 u(1+\delta/L_1) + A_2 L_2 u(1+\delta/L_2) .
\end{split}
\end{equation}

We begin by ascertaining conditions on $M$ for separability. Let $K = [k_{\rm min},k_{\rm max}]$, $0<k_{\rm min}<1$, $k_{\rm max}>1$. Assume $L_1 \neq L_2$. {The set of possible displacements $\delta$ is characterized by $1+\delta/L_1\in K$ and $1+\delta/L_2\in K$, which means that 
\begin{equation}
    \delta\in M
    =
    [(k_{\rm min}-1)\max(L_1,L_2),(k_{\rm max}-1)\min(L_1,L_2)].
\end{equation}
Since $(k_{\rm min}-1)\max(L_1,L_2)<0<(k_{\rm max}-1)\min(L_1,L_2)$, this set has nonempty interior.} Since $E(u;\delta)$ depends linearly on $u$, to prove that $M$ is separating it suffices to show that $E(w;\delta)=0$ for all $\delta\in M$ and $w=v-u$, $u,v \in U$, implies that $w=0$. Assume $A_1=A_2$, for simplicity. Then, it follows from (\ref{z8ezZq}) that $E(w;\delta)=0$ if 
\begin{equation} \label{QgF7Rj}
    L_1 w(1+\delta/L_1) + L_2 w(1+\delta/L_2) = 0.
\end{equation}
Fix $\delta\in M$ and let $\beta = \delta/\min(L_1,L_2)$. Then, 
(\ref{QgF7Rj}) becomes 
\begin{equation} \label{yA6svg}
    w(1+\lambda \beta) + \lambda w(1+\beta) = 0
\end{equation}
where
\begin{equation}
    \lambda 
    = 
    \frac{\min(L_1,L_2)}{\max(L_1,L_2)}
    <
    1 .
\end{equation}
{Since $M$ is an interval containing zero, $\delta\in M$ implies $\lambda 
\delta \in M$.} Iterating (\ref{yA6svg}) we obtain 
\begin{equation} 
    w(1+\beta) = (-\lambda)^{-n} w(1+\lambda^n \beta) ,
\end{equation}
or, taking absolute values,
\begin{equation} 
    | w(1+\beta) | = \lambda^{-n} \, |  w(1+\lambda^n \beta) | .
\end{equation}
Suppose further that, in a neighborhood of $\xi=1$, 
\begin{equation}\label{eqasspgr}
    |  w(\xi) | \leq C \, | \xi - 1 |^p ,
\end{equation}
for some $C>0$, $p>1$. Then,
\begin{equation} 
    | w(1+\beta) | 
    \leq 
    \lim_{n\to\infty} C \, |\beta|^p \lambda^{n (p-1)} 
    =
    0 ,
\end{equation}
whence we conclude that $M$ is separating.

\section{Approximation}

The preceding variational framework can be exploited to formulate 
approximation schemes with provable convergence properties. In this section 
we consider successive approximations to material identification problems. We 
show that, under natural assumptions, subsequences of  constrained optima 
indeed converge to the sought 'ground truth' energy density. 

\subsection{A Galerkin scheme}

We consider approximations obtained by restriction of $J(u)$ to dense 
subspaces of energy densities, i.~e., by minimizing 
\begin{equation} \label{5MgEnG}
    J_h(u)
    =
\begin{cases}
    J(u), & \text{\rm if } u \in U_h, \\
    + \infty, & \text{\rm otherwise,}
\end{cases}
\end{equation}
over some suitable sequence of closed subspaces $U_h \subset U$. Next we show 
that the convergence of the corresponding approximations can be characterized 
in terms of $\Gamma$-convergence \cite{Conti:2007}. 

We denote by $(U,T)$ the space $U$ endowed with its strong topology $T$ and a 
by $(U,S)$ the space $U$ endowed with its weak topology $S$. We additionally 
denote by $\to$ strong convergence and by $\rightharpoonup$ weak convergence. 
{In the current setting, both are metrizable.}

\begin{definition}[Density]
We say that a sequence $(U_h)$ has the density property if for every $u \in 
U$, there is a sequence $(u_h)$, with $u_h \in U_h$, such that $u_h \to u$ 
{in the strong topology of $U$}. 
\end{definition}

In the context of neural-network approximation, dense sequences of spaces are
said to have a {\sl universal approximation property}. 

\begin{prop} [$\Gamma$-convergence] \label{KzmpW7}
Let $J: U \to \overline{\mathbb{R}}$ be coercive in $(U,S)$ and continuous in 
$(U,T)$. Let $(U_h)$ be a sequence of closed subspaces of $U$ with the 
density property, and $(J_h): U \to \overline{\mathbb{R}}$ be the sequence 
(\ref{5MgEnG}). Then the sequence $(J_h)$ $\Gamma$-converges to ${\rm sc}^-J$ 
in $(U,S)$ and is equicoercive in $(U,S)$. 
\end{prop}

\noindent 
\begin{proof} 
i) $\liminf$-inequality. Let $u\in U$, and $u_h\to u$ with respect to $S$. 
Since $J_h \geq J$ on $U$ for every $h \in \mathbb{N}$, we have 
\begin{eqnarray}
    {\rm sc}^- J(u)
    \le
    \liminf_{h\to\infty} J(u_h)
    \le
    \liminf_{h\to\infty} J_h(u_h) \,,
\end{eqnarray}
as required. 

ii) $\limsup$-inequality. Let $u\in U$. By (\ref{defRel}) there is a sequence 
$u^k\rightharpoonup u$ in $S$ such that 
\begin{equation}
    \lim_{k\to\infty} J(u^k)
    =
    {\rm sc}^-J(u) \,.
\end{equation}
By the density of the sequence $U_h$, for any $k$ there is a sequence 
$u_h^k\in U_h$, with $u_h^k\to u^k$ with respect to $T$. Since $J$ is 
continuous in the same topology, and $J_h=J$ on $U_h$, 
\begin{equation}
    \lim_{h\to\infty} J_h(u_h^k)
    =
    \lim_{h\to\infty} J(u_h^k) = J(u^k) \,.
\end{equation}
The $\limsup$ inequality then follows by passing to a diagonal subsequence. 

iii) Equicoercivity. Since, by construction, $J_h\ge J$, the equicoercivity 
of the sequence is immediate from the coercivity of $J$. 
\end{proof}

As noted earlier, $\Gamma$-convergence and equicoercivity imply that the 
minimizers of $J$ are accumulation points of minimizing sequences of the 
family $J_h$, i.~e., if $J_h(u_h)=\inf J_h$ then the sequence $u_h$ has a 
subsequence that converges weakly to a minimizer of $J$. 

We also note that the existence of minimizers $u_h$ of $J_h$ follow 
immediately from the closedness of $U_h$, which ensures that $J_h$ inherits 
the lower-semicontinuity of $J$, and definition (\ref{5MgEnG}), which ensures 
that $J_h$ inherits uniformly the coercivity of $J$ (equicoercivity). 

\subsection{Approximation by maxout neural networks}

We specialize the preceding approximation scheme to sequences $(U_h)$ 
generated by {\sl maxout} neural networks. This choice of approximating 
energy functions is suggested by the following structural properties of $U$. 

\begin{prop}[Lattice structure of $U$] \label{DMzS7G}
Let $u_i \in U$, $i\in I$. Assume that $u = \bigvee_{i\in I} u_i$ is not 
identically $+ \infty$. Then $u\in U$. 
\end{prop}

\begin{proof}
By the definition of $U$ there are convex, $\ell$-Lipschitz functions
$g_i:\R^{\tau(n)}\to\R$ such that $u_i=g_i\circ{\rm Minors}$. As $u$ is not 
identically $+\infty$ there is a point $\xi_0\in K$ such that $u_i(\xi_0)$ is 
uniformly bounded from above, therefore the sequence $g_i$ is uniformly 
bounded from above on the point $m_0={\rm Minors}(\xi_0)$. As the functions 
$g_i$ are uniformly Lipschitz, they are locally uniformly bounded from above. 
Therefore $g=\bigvee_{i\in I} g_i$ defines a function $g:\R^{\tau(n)}\to\R$. 
Obviously $u=g\circ{\rm Minors}$. 

Let $\xi$, $\xi'\in \R^{\tau(n)}$. For any $\epsilon>0$ there is 
$i\in I$ with $g(\xi)<g_i(\xi)+\epsilon$. Then
\begin{equation}
 g(\xi)-g(\xi')\le g_i(\xi)+\epsilon -g(\xi')
 \le  \ell |\xi-\xi'|+\epsilon + g_i(\xi')-g(\xi')
 \le \ell |\xi-\xi'|+\epsilon.
\end{equation}
As $\epsilon$ was arbitrary,  the same holds for $\epsilon=0$. Since $\xi$ 
and $\xi'$ were arbitrary, this proves that $g$ is $\ell$-Lipschitz. 

Finally, as the supremum of convex functions is convex, $g$ is convex.
Hence, $u \in U$, as advertised. 
\end{proof}

Recall that a function $f:\mathbb{R}^{n\times n}\to\mathbb{R}$ is said to be 
polyaffine if $f$ and $-f$ are  polyconvex, see 
\cite[Def.~1.5]{Dacorogna:2007}. Similar concepts exist for quasiconvexity 
and rank-one convexity. We remark that a function is quasiaffine if and only 
if it is polyaffine (see for example 
\cite[Th.~6.1]{ContiDolzmannKirchheimMueller2006}). {In turn, this is 
equivalent to $f$ being an affine function of the minors, so that 
$f=g\circ{\rm Minors}$ with $g:\R^{\tau(n)}\to\R$ affine.} 

In view of the lattice property of $U$, Prop.~\ref{DMzS7G},
we specifically consider sequences of approximating spaces of the form
\begin{equation} \label{2F7dmS}
    U_h 
    = 
    \{ 
        u \in U \,:\, 
        u = \bigvee_{i=1}^{N_h} f_i ,\;
        f_i \in U,\; \text{polyaffine},\; i=1,\dots,N_h
    \} .
\end{equation}
We note that the spaces $U_h$ are finite-dimensional and, hence, closed in 
$U$. {Indeed, to characterize each $u\in U_h$ it is sufficient to 
characterize at most $N_h$ polyaffine functions $f_i$, with $\{u=f_i\}\cap K$ 
nonempty for all $i$. This, together with the fact that $f_i\in U$, implies 
that each $f_i$ can be characterized by finitely many real coefficients, 
which have a uniform bound (depending on a uniform bound on $u$).} 
\begin{rem}[Neural network interpretation]\label{7mp7wM}{\rm 
We note that any $u_h \in U_h$ may be regarded as a neural network in which 
the affine functions $f_i$ are the neurons and $\bigvee$ is the activation 
function. This type of neural network has been studied, e.~g., in 
\cite{Goodfellow:2013, Anil:2019}, where they are referred to as {\sl maxout} 
networks. We note that multilayer neural networks can also be defined by 
iterating (\ref{2F7dmS}). However, the multilayer neural networks thus 
defined reduce to a single layer neural network of the form (\ref{2F7dmS}) by 
the associativity of the wedge activation function. This reduction has the 
convenient consequence that no regard needs to be taken of the 'topology' of 
the neural networks, as they all reduce to single layer networks. 
\hfill$\square$} 
\end{rem}

\begin{prop}[Density] \label{P8bAkX}
The sequence $(U_h)$ defined in (\ref{2F7dmS}) has the density property. 
\end{prop}

\begin{proof}
Fix a function $u\in U$.
Let $g:\mathbb R^{\tau(n)}\to\R$ be convex, $\ell$-Lipschitz, and such that 
$u=g\circ\mathrm{Minors}$ on $K$.

For every $\xi\in K$, let $h_\xi$ denote an affine map which is tangent to 
$g$ in the point $m=\mathrm{Minors}(\xi)\in\R^{\tau(n)}$ (this is unique if 
$g$ is differentiable at that point).  The function $h_\xi$ is automatically 
$\ell$-Lipschitz, and $h_\xi\le g$ everywhere, therefore we can define a 
function $u_\xi=h_\xi\circ\mathrm{Minors}\in U$ with $u_\xi\le u$. 

Fix a sequence $\delta_h\to0$. For every $h$ one can choose finitely many 
points $\xi_1,\dots, \xi_{N_h}\in K$ such that for any $\xi\in K$ there is 
$i$ with $|\xi-\xi_i|<\delta_h$ (this corresponds to the condition 
$K\subseteq \cup_i B_{\delta_h}(\xi_i)$). Then 
\begin{equation}
    u_h:=\bigvee_{i=1}^{N_h} u_{\xi_i} \in U_h
\end{equation}
and $u_h\le u$ on $K$. We observe that these functions are uniformly 
Lipschitz. Indeed, the functions $h_{\xi_i}$ are $\ell$-Lipschitz, and the 
polynomial  $\mathrm{Minors}:K\to \R$ is Lipschitz on the bounded set $K$. 
Therefore for any $\xi\in K$, choosing $\xi_i$ as above, one has 
\begin{equation}
    |u_h(\xi)-u(\xi)|\le
(\mathrm{Lip}(u_h)+\mathrm{Lip}(u))|\xi-\xi_i|
\le 2\ell \mathrm{Lip}(\mathrm{Minors}|_K) \delta_h.
    \end{equation}
This proves that $u_h\to u$ uniformly.
\end{proof}

The following corollary now follows immediately from Props.~\ref{P8bAkX} 
and~\ref{KzmpW7} and the properties of $\Gamma$-convergence.

\begin{cor}[Convergence of maxout neural network approximations]
Let $(U_h)$ be as in (\ref{2F7dmS}) and $(J_h)$ as in (\ref{5MgEnG}). 
Let $(u_h)$ be a sequence of minimizers of $(J_h)$ in $U$. 
Then, there is $u \in U$ such that $u_h \rightharpoonup u$. 
\end{cor}

\section{Examples} \label{Jz8Rgr}

Problem (\ref{SC3dBH}) has the structure of a {\sl minimax problem} 
\cite{Ekeland:1999}, also known as a Chebyshev approximation problem. It can 
be equivalently expressed as the nonlinear optimization program 
\begin{subequations} \label{g9a6Ch}
\begin{align}
    &
    \operatorname{\min} z ,
    \\ &
    z - I(u;g) \geq 0 , \quad g \in M ,
\end{align}
\end{subequations} 
with $I(u;g)$ as in (\ref{VgsvM1}), which makes contact with the vast body of 
work on nonlinear optimization. The connection to maxout networks, 
Remark~\ref{7mp7wM}, further enables the application of powerful tools from 
machine learning to the identification of general three-dimensional material 
laws. However, a detailed consideration of matters of numerical 
implementation is beyond the scope of this paper (the interested reader is 
referred to \cite{Conti:2024}). Instead, in this section we present simple 
examples that help visualize the tradeoffs and approximation strategies set 
forth in the foregoing. 

\subsection{One bar} \label{4EAuzh}

Consider a bar as in Example~\ref{UMGK5s}. In this case, the approximation 
spaces (\ref{2F7dmS}) reduce to 
\begin{equation} \label{8pgSV5}
    U_N 
    = 
    \{ 
        u = \bigvee_{i=1}^{N} f_i ,\;
        f_i = a_i \xi + b_i,\; |a_i| \leq C_1, \; i=1,\dots,N
    \} .
\end{equation}
We seek minimizers of the restricted cost functional $J_N(u)$, 
eq.~(\ref{5MgEnG}). Typical test functions $u_N \in U_N$ are shown in 
Fig.~\ref{c9agZ3}a. 

\begin{figure}[ht!]
\begin{center}
    \subfigure[]{\includegraphics[width=0.49\textwidth]{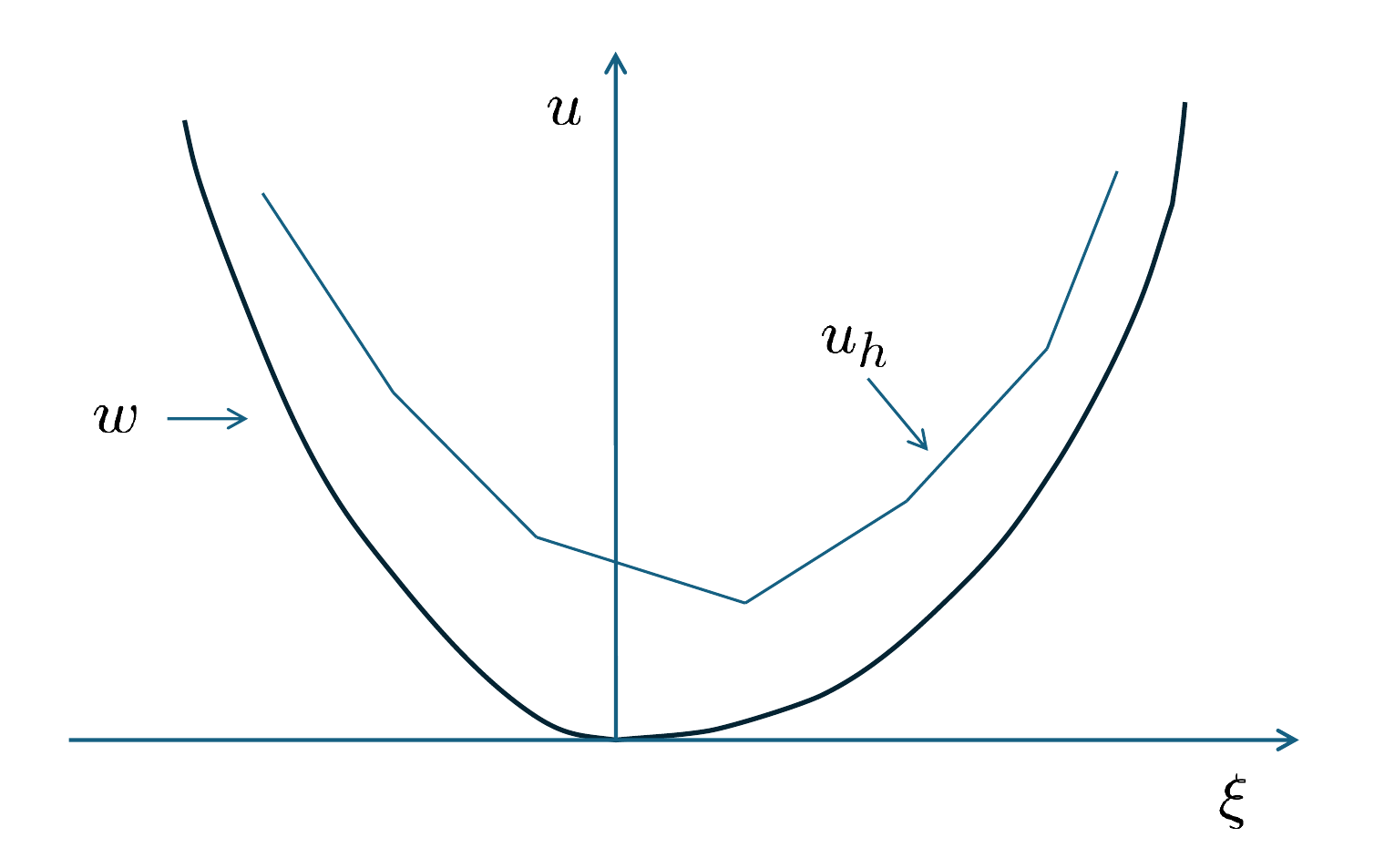}}
    \subfigure[]{\includegraphics[width=0.49\textwidth]{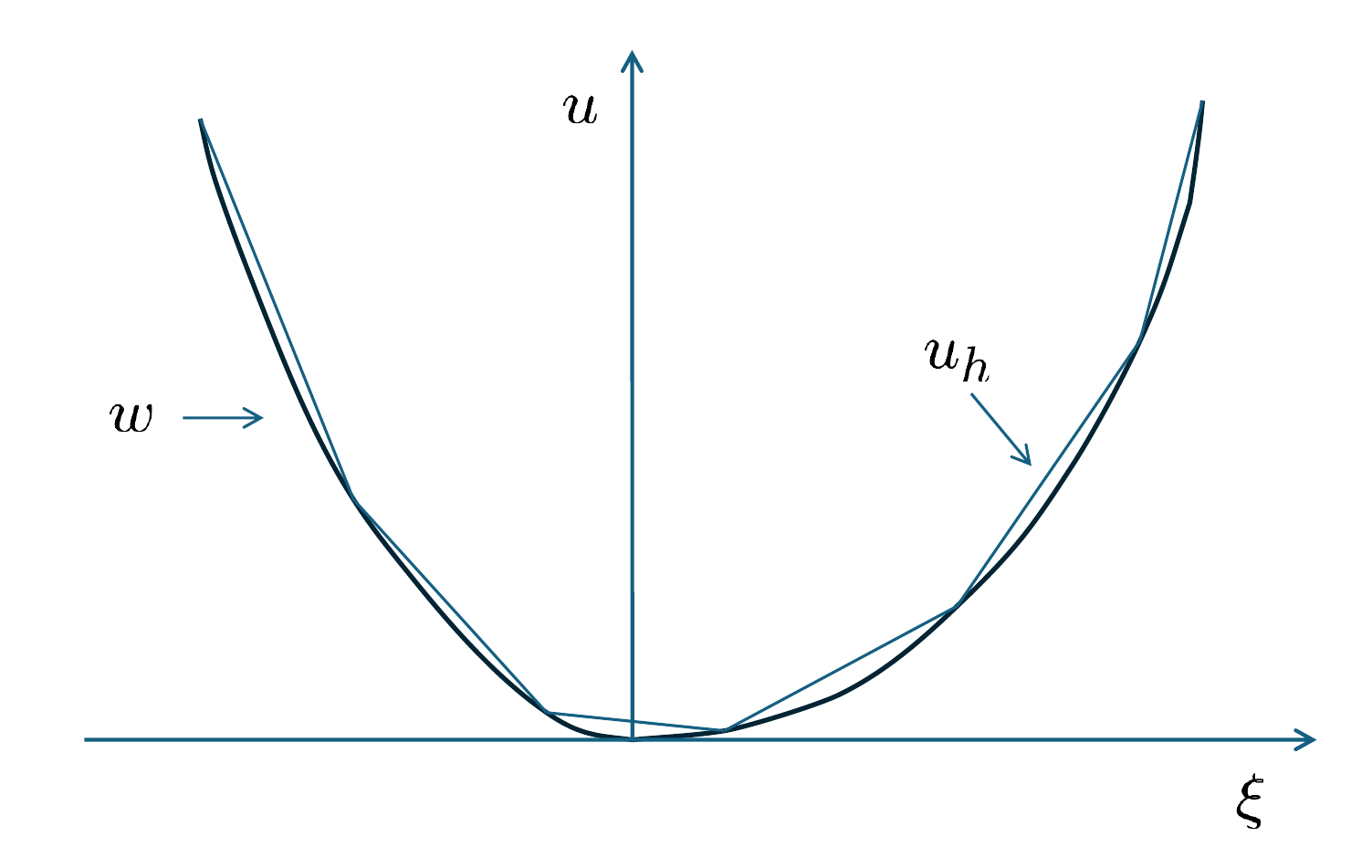}}
    \caption{One-dimensional bar. a) Approximating piece\-wise-affine 
    functions. b) Extremal approximating piece\-wise-affine 
    functions.} \label{c9agZ3} 
\end{center}
\end{figure}

\subsubsection{Linear elasticity} \label{CCKg3B}

As a particular example, suppose that the unknown energy density is 
\begin{equation} \label{9T7KFa}
    w(\xi) = \frac{\mathbb{C}}{2} \epsilon^2 ,
    \quad
    \xi = 1 + \epsilon ,
\end{equation}
where $\mathbb{C}$ is the elastic modulus and $\epsilon$ is the strain. We 
additionally suppose that the bar is tested over a set of prescribed 
displacements $\delta$ in some interval $M = [\delta_{\rm min}, \delta_{\rm 
max}]$. For fixed $\delta \in M$, the minimizing strains are uniform over the 
length of the bar and the minimum energy follows as 
\begin{equation} \label{JNTq3u}
    E_0(\delta)
    =
    A L \frac{\mathbb{C}}{2} (\delta/L)^2 ,
\end{equation}
where $A$ is the cross-sectional area and $L$ is the length of the bar. 

We consider trial energy densities of the form 
\begin{equation} \label{WF5HvP}
    u_N(\epsilon)
    =
    \max_{1\leq i \leq N} \{a_i + b_i \epsilon\} ,
\end{equation}
adapted from (\ref{8pgSV5}) to the linearized kinematics under consideration. 
The corresponding minimum energy of the bar is, then, 
\begin{equation}
    E(u_N;\delta)
    =
    A L \max_{1\leq i \leq N} \{a_i + b_i \delta/L\} .
\end{equation}
In keeping with (\ref{VgsvM1}), we restrict the trial densities so as to 
satisfy the bound 
\begin{equation} \label{L22Sds}
    E_0(\delta) \leq E(u_N;\delta), \quad \forall \delta \in M.
\end{equation}
The cost function (\ref{ptMeIB}) then follows as 
\begin{equation} \label{n3tHvK}
\begin{split}
    &
    J(u_N)
    = 
    \max_{\delta \in M}
    \min_{1\leq i \leq N} 
    \{E_0(\delta) - A L ( a_i + b_i \delta/L)\}) ,
\end{split}
\end{equation}
to be minimized with respect to $(a_i,b_i)_{i=1}^N$ subject to the 
constraint (\ref{L22Sds}). 

Evidently, by convexity the minimizing test functions rest on the graph of 
$w(\xi)$, Fig.~\ref{c9agZ3}b, on an ordered point set $(\delta_0, \dots, 
\delta_N)$, with $\delta_0 = \delta_{\rm min}$ and $\delta_N = \delta_{\rm 
max}$. In every interval $(\delta_i,\delta_{i+1})$, we have 
\begin{equation}
    \max\{ E(u_N;\delta) - E_0(\delta) \,:\, 
    \delta \in (\delta_i,\delta_{i+1}) \} 
    = 
    \frac{\mathbb{C} A}{8 L} (\delta_{i+1}-\delta_i)^2 ,
\end{equation}
and, maximizing over all intervals, we obtain 
\begin{equation} \label{7jWkek}
    J(u_N)
    = 
    \max_{0\leq i \leq N-1}
    \frac{\mathbb{C} A}{8 L} (\delta_{i+1}-\delta_i)^2 .
\end{equation}

The cost function (\ref{7jWkek}) is now to be minimized with respect to the 
point set $(\delta_i)_{i=1}^{N-1}$, which defines a {\sl minimax} problem. 
Alternatively, we can reformulate this problem as the ordinary convex program 
\cite[\S28]{Rockafellar:1970} 
\begin{subequations} \label{7v2Qdv}
\begin{align}
    &
    \min z =: f_0(x),
    \\ & \label{uFh672}
    f_i(x) 
    := 
    \frac{\mathbb{C} A}{8 L} (\delta_i-\delta_{i-1})^2 - z \leq 0 ,
    \quad i = 1,\dots,N ,
    \\ & \label{QG9MKb}
    (\delta_1,\dots,\delta_{N-1})
    \in
    C := \{\delta_{\rm min} \leq \delta_1 \leq 
    \cdots 
    \leq \delta_{N-1} \leq \delta_{\rm max} \},
\end{align}
\end{subequations}
where we write $x:=(z,\delta_1,\dots,\delta_{N-1})$. Define the Lagrangian, 
\begin{equation} \label{gfE5bq}
    L(x,\lambda)
    =
    \left\{
    \begin{array}{ll}
        f_0(x) + \lambda_1 f_1(x) + \cdots + \lambda_N f_N(x) ,
        & \text{if} \;\; \lambda \in E, \; x \in C , \\
        -\infty, & \text{if} \;\; \lambda \not\in E, \; x \in C , \\
        +\infty, & \text{if} \;\; x \not\in C ,
    \end{array}
    \right.
\end{equation}
where 
\begin{equation} \label{pfy5VB}
    E 
    = 
    \{ 
        \lambda_i \geq 0, \; i = 1, \dots, N 
    \} .
\end{equation}
By the Kuhn-Tucker theorem \cite[Corollary 28.3.1]{Rockafellar:1970}, the 
solutions of the convex program satisfy the conditions 
\begin{subequations} \label{D7taZJ}
\begin{align}
    &   \label{gfE5bqb}
    \lambda_i \geq 0 ;\;
    f_i({x}) \leq 0 ;\;
    \lambda_i f_i({x}) = 0, \;\; i = 1, \dots, N,
    \\ & \label{yrg5Xu}
    \partial f_0({x}) 
    + 
    \lambda_1 \partial f_1({x}) 
    + \dots + 
    \lambda_N \partial f_N ({x})
    =
    0 .
\end{align}
\end{subequations}
By simple inspection, we find that these conditions are satisfied by 
\begin{equation} 
    \delta_i^* - \delta_{i-1}
    =
    \frac{\delta_{\rm max}-\delta_{\rm min}}{N} ,
    \quad
    \lambda_i
    =
    \frac{1}{N} ,
    \quad
    i = 1,\dots,N .
\end{equation}
We note that $\lambda_i > 0$, whence the Kuhn-Tucker conditions 
(\ref{D7taZJ}) require 
\begin{equation} \label{e8D4Dz}
    f_1(x) = \dots = f_N(x) = 0,
\end{equation}
corresponding to an equi-distribution of the error. 

We thus conclude that the optimal point set $(\delta_i^*)_{i=0}^N$ 
corresponds to a uniform partition of $M$. The optimal prescribed 
displacements $(\delta_i^*)_{i=0}^N$ may be regarded as defining a program of 
{\sl optimal tests}. It bears emphasis that a finite number of such tests is 
required at every level of approximation. Solutions for two and five neurons 
are depicted in Fig.~\ref{MWWnL8}, by way of illustration. 

\begin{figure}[ht!]
\begin{center}
    \subfigure[]{\includegraphics[width=0.49\textwidth]{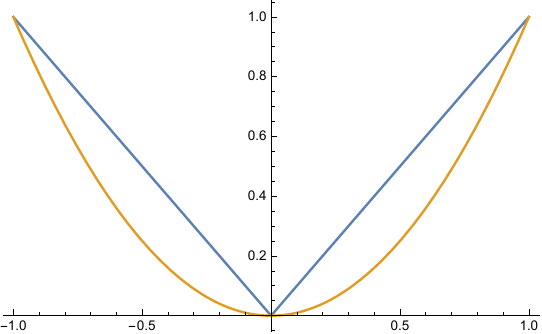}}
    \subfigure[]{\includegraphics[width=0.49\textwidth]{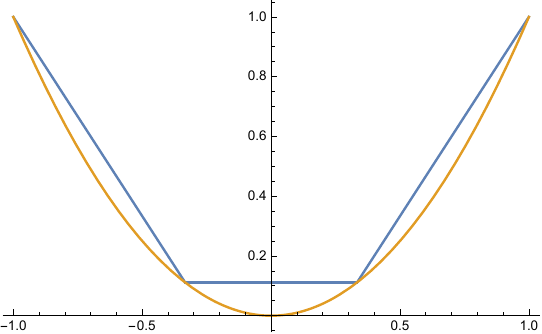}}
    \subfigure[]{\includegraphics[width=0.49\textwidth]{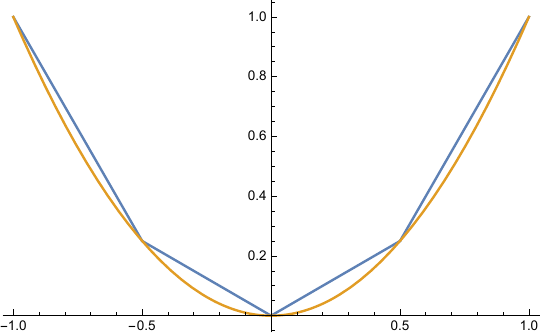}}
    \subfigure[]{\includegraphics[width=0.49\textwidth]{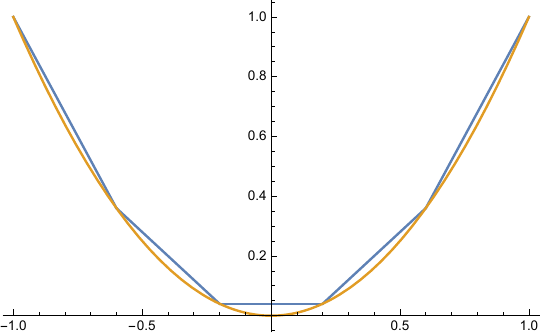}}
    \caption{One-dimensional bar, linearized kinematics, eq.~(\ref{9T7KFa}). 
    Optimal successive approximations generated by the tangent construction. 
    a) Two-neuron approximation; b) Three-neuron approximation; 
    c) Four-neuron approximation; d) Five-neuron approximation.} 
    \label{MWWnL8} 
\end{center}
\end{figure}

\subsubsection{Finite elasticity} \label{Ccvcy9}

We wish to identify a material which, unbeknownst to the tester, is 
incompressible Neo-Hookean, with energy density 
\begin{equation}\label{VZNwbO}
    w(\xi_1, \xi_2, \xi_3) 
    = 
    \left\{
    \begin{array}{ll}
        \dfrac{\mu}{2} \big( \xi_1^2 + \xi_2^2 + \xi_3^2 - 3 \big) ,
        & \text{if} \;\; \xi_1 \xi_2 \xi_3 = 1, \\ [0.25truecm]
        +\infty , & \text{otherwise} ,
    \end{array}
    \right.
\end{equation}
where $(\xi_1, \xi_2, \xi_3)$ are the principal stretches and $\mu > 0$ is 
the shear modulus. Suppose that the specimen takes the form of a round bar, 
$0 < x < L$, of constant circular cross section of area $A$. In order to 
simplify the problem, we resort to one-dimensional rod theory
\begin{subequations}\label{n3GjE4}
\begin{align}
    &
    G(u,y; g)
    = 
    \int_0^L A \, u(y'(x)) \, dx ,
    \qquad
    y(0) = 0 , \quad y(L) = g ,
\end{align}
\end{subequations}
where $u$ is a trial energy density and $y:(0,L)\to\mathbb{R}$ is a trial 
axial deformation mapping,
and focus on identifying the reduced energy density 
\begin{equation}\label{HnSZSa}
    w(\xi)
    =
    \frac{\mu}{2} \Big( \frac{2}{\xi} + \xi^2 - 3 \Big) .
\end{equation}
The bar is tested over a set of prescribed end 
displacements $g$ in some interval $M = (g_{\rm min}, g_{\rm max}) \subset 
(0,+\infty)$. For fixed $g \in M$, the minimizing deformations are uniform 
over the length of the bar and the minimum energy follows as 
\begin{equation} \label{Z67v8f}
    E_0(g)
    =
    A L \frac{\mu}{2} 
    \Big( \frac{2 L}{g} + \frac{g^2}{L^2} - 3 \Big) .
\end{equation}

Proceeding as in the linear-elastic case, with 
\begin{equation} \label{wseFH9}
    u_N(\xi)
    =
    \max_{1\leq i \leq N} \{a_i + b_i \xi\} ,
\end{equation}
we arrive at an ordinary convex program of the form (\ref{7v2Qdv}) with 
\begin{equation}
    f_i(x) := \max \{ E(u_N,g) - E_0(g) \,:\, g \in (g_{i-1},g_i) \},
    \quad
    i = 1,\dots,N ,
\end{equation}
and $x := (z,g_1,\dots,g_N)$. The corresponding Kuhn-Tucker conditions are 
again of the form (\ref{e8D4Dz}), which determines the optimal testing 
program $(g_i^*)_{i=1}^N$. Again we observe that only a finite number of 
tests is required at every level of approximation. Optimal solutions for two 
to five neurons are depicted in Fig.~\ref{4zwMyX}, by way of illustration. 

\begin{figure}[ht!]
\begin{center}
    \subfigure[]{\includegraphics[width=0.49\textwidth]{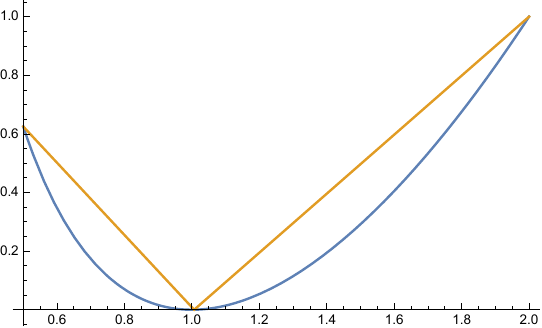}}
    \subfigure[]{\includegraphics[width=0.49\textwidth]{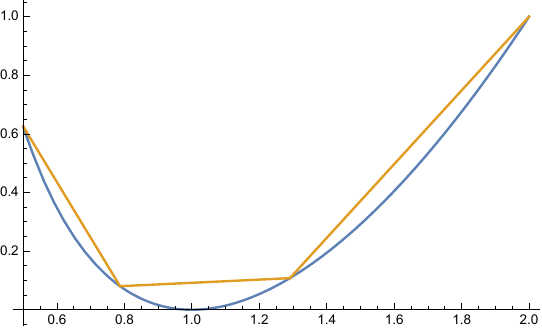}}
    \subfigure[]{\includegraphics[width=0.49\textwidth]{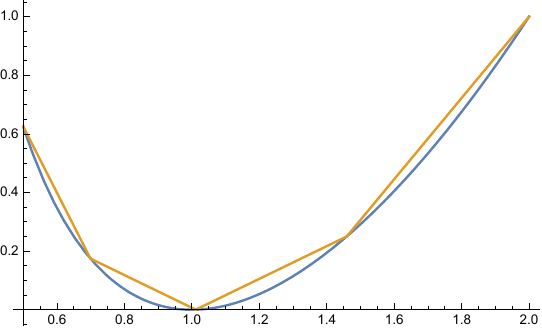}}
    \subfigure[]{\includegraphics[width=0.49\textwidth]{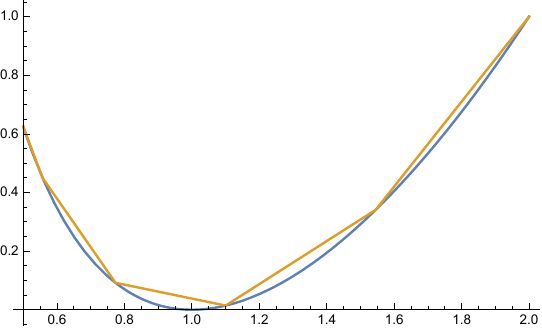}}
    \caption{One-dimensional bar, finite kinematics, eq.~(\ref{HnSZSa}), 
    successive optimal approximations. a) Two-neuron approximation; b) 
    Three-neuron approximation; c) Four-neuron approximation; d) Five-neuron 
    approximation.} 
    \label{4zwMyX} 
\end{center}
\end{figure}

\subsection{Two bars in parallel}

The preceding one-bar examples illustrate the minimax character of the 
identification problem using maxout neural-network approximation. However, 
they are special in that the state of the system is characterized by one 
single deformation that is in one-to-one correspondence to the prescribed 
displacements. The two-bar Example~\ref{9AqoKC} provides a simple 
illustration of a system in which a number of deformations is induced for 
every applied displacement and the measured total energy conflates the 
corresponding local energy densities. 

We assume that the underlying material obeys the one-dimensional Neo-Hookean 
law (\ref{HnSZSa}). For fixed prescribed displacement $\delta \in M$, the 
minimizing deformations are uniform over the length of each bar and the 
minimum energy follows as 
\begin{equation} 
    E_0({ \delta})
    =
    \sum_{k=1}^2
    A_k L_k \frac{\mu}{2} 
    \Big( \frac{2}{\xi_k} + \xi_k^2 - 3 \Big) ,
    \quad
    \xi_k := 1 + \delta/L_k.
\end{equation}
Likewise, the minimum energy for a trial energy density of the form 
(\ref{wseFH9}) evaluates to 
\begin{equation} 
    E(u_{N};{ \delta})
    =
    \sum_{k=1}^2
    A_k L_k \max_{1\leq i\leq N} (a_i+b_i \xi_k) ,
\end{equation}
which is convex and piecewise linear consisting of possibly up to $2N$ 
segments. The optimal approximations follow from a ordinary convex program 
similar to those analyzed in the preceding sections, which can be solved 
directly for the unknowns $(a_i,b_i)_{i=1}^N$. 

More directly, a minimizing sequence can be constructed directly from the 
one-bar solutions of Section~\ref{Ccvcy9}. To this end, let $U_N$ be defined 
according to (\ref{2F7dmS}) and let 
\begin{equation} \label{CvyZM9}
    u_N^*(\xi)
    =
    \max_{1\leq i \leq N} \{a_i^* + b_i^* \xi\} ,
    \quad
    \xi = \frac{1+\delta}{\min(L_1,L_2)} ,
    \quad \delta \in M ,
\end{equation}
be locally optimal in the sense 
\begin{equation} \label{pA2ANS}
    \max_{\xi \in K}
    \Big(
        w(\xi) 
        - 
        u_N^*(\xi)
    \Big)
    \leq 
    \max_{\xi \in K}
    \Big(
        w(\xi) 
        - 
        u_N(\xi)
    \Big) ,
    \quad
    w \leq u_N \in U_N .
\end{equation}
We claim that $u_N^*$ is a minimizing sequence for the identification problem. 

To prove this claim, we begin by noting that, if $u \geq w$, then 
\begin{equation} \label{dF7LHY}
    E(u;{ \delta})
    =
    \sum_{k=1}^2 
        A_k L_k  u(\xi_k)
    \geq
    \sum_{k=1}^2 
        A_k L_k w(\xi_k)
    =
    E_0({ \delta}) ,
\end{equation}
for all {${ \delta} \in M$}, and $u \in C$. 

Next, we construct a test function $f_N$ such that $J(f_N)\to0$. To do this,
subdivide $K=[k_{\rm min},k_{\rm max}]$ into $N$ intervals of equal
length, with $\xi_0=k_{\rm min}<\xi_1<\dots\xi_N=k_{\rm max}$ denoting the 
nodes. Let $f_N:K\to\R$ be a piecewise affine function that coincides with 
$w$ on those nodes. Since $w$ is convex, $w\le f_N$ on $K$. Furthermore, the 
Lipschitz constant of $f_N$ is bounded by the Lipschitz constant of $w$, and 
the maximum of $f_N$ over $K$ is bounded by the maximum of $w$ over $K$. 
Therefore, $f_N\in U$, and, since it is piecewise affine, $f_N\in U_N$. We 
have the bound 
\begin{equation} \label{4CfYC3}
    \|f_N-w\|_{L^\infty(K)}
    \le 
    \Lip(w) \, \frac{k_{\rm max}-k_{\rm min}}N ,
\end{equation} 
and, if $w\in C^2(K)$, we further have 
\begin{equation}
    \|f_N-w\|_{L^\infty(K)}
    \le 
    \|w\|_{C^2(K)} \, \frac{(k_{\rm max}-k_{\rm min})^2}{N^2} .
\end{equation}
Recall from Lemma \ref{lemmaUC0} that 
\begin{equation} \label{jy2aCW}
    J(f_N) \le J(w) + 
    (A_1 L_1 + A_2 L_2) \,
    \|f_N-w\|_{L^\infty(K)} ,
\end{equation}
which is a special case of (\ref{esEuu1}). Since $J(w)=0$, (\ref{4CfYC3}) and 
(\ref{jy2aCW}) then jointly imply that $J(f_N)\to 0$ as $N\to\infty$. 

But
\begin{equation} \label{kYv2Bk}
\begin{split}
    & J(f_N) \ge 
    \min \{ J(u_{{N}}) : u_N\in U_N, u_N\ge w\}
    = \\ &
    \min_{u_N \geq w}
    \max_{\delta \in M}
    \Big\{
        \sum_{k=1}^2 A_k L_k 
        \Big(
            \max_{1\leq i\leq N} (a_i+b_i \xi_k)
-             w(\xi_k) 
        \Big)
    \Big\}
    \ge \\ & 
    \max_{\delta \in M}
    \min_{u_N \geq w}
    \Big\{
        \sum_{k=1}^2 A_k L_k 
        \Big(
            \max_{1\leq i\leq N} (a_i+b_i \xi_k)
            - 
            w(\xi_k)
        \Big)
    \Big\}
    \geq \\ & 
    \max_{\delta \in M}
    \Big\{
        \sum_{k=1}^2 A_k L_k 
        \min_{u_N \geq w}
        \Big(
            \max_{1\leq i\leq N} (a_i+b_i \xi_k)
            - 
            w(\xi_k) 
        \Big)
    \Big\}
    \geq \\ & 
    \max_{\delta \in M}
    \Big\{
        \sum_{k=1}^2 A_k L_k 
        \Big(
            \max_{1\leq i\leq N} (a^*_i+b^*_i \xi_k)
            - 
            w(\xi_k) 
        \Big)
    \Big\}
    =
    J(u_N^*) , 
\end{split}
\end{equation}
which shows also $J(u_N^*)$ converges to zero and $u_N^*$ is a minimizing 
sequence for the optimal identification problem, as advertised.  

The global energy functions $E(u_N^*;{ \delta})$ for a two-bar structure 
resulting from the optimal solutions shown in Fig.~\ref{4zwMyX} are collected 
in Fig.~\ref{6gDRmF} by way of illustration. 

\begin{figure}[ht!]
\begin{center}
    \subfigure[]{\includegraphics[width=0.49\textwidth]{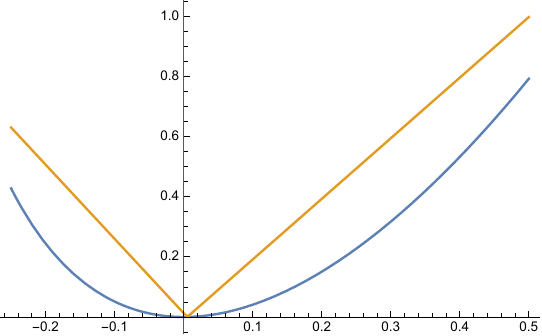}}
    \subfigure[]{\includegraphics[width=0.49\textwidth]{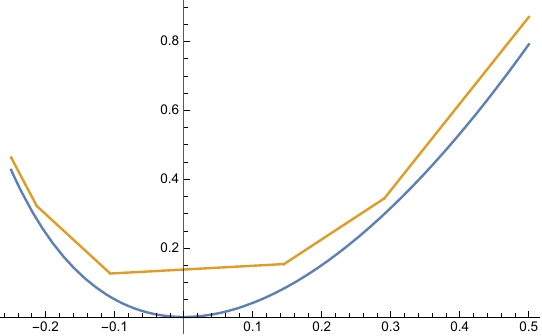}}
    \subfigure[]{\includegraphics[width=0.49\textwidth]{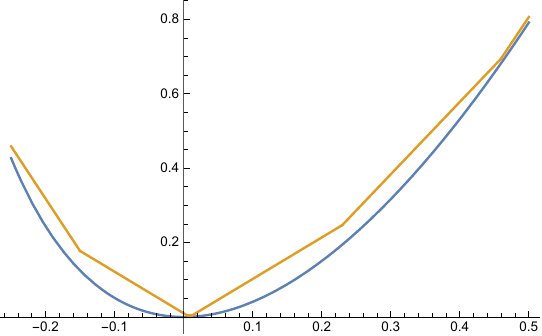}}
    \subfigure[]{\includegraphics[width=0.49\textwidth]{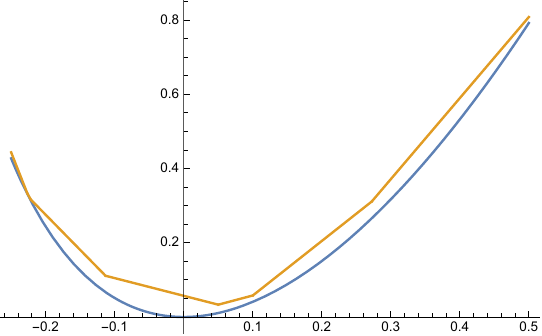}}
    \caption{One-dimensional bar, finite kinematics, eq.~(\ref{HnSZSa}). 
    Successive global energy functions $E(u_N^*;g)$ corresponding to the 
    minimizing sequence (\ref{CvyZM9}) and (\ref{pA2ANS}). 
    a) Two-neuron approximation; b) Three-neuron approximation; 
    c) Four-neuron approximation; d) Five-neuron approximation.} 
    \label{6gDRmF}
\end{center}
\end{figure}

\section{Summary and concluding remarks}

We have formulated the problem of material identification as a problem of 
optimal control in which the deformation of the specimen is the state 
variable and the unknown material law is the control variable. We assume that 
the material obeys finite elasticity and the deformation of the specimen is 
in static equilibrium with prescribed boundary displacements. We further 
assume that the attendant total energy of the specimen can be measured, 
e.~g., with the aid of the work-energy identity. In particular, no full-field 
measurements, such as DIC, are required. The cost function measures the 
maximum discrepancy between the total elastic energy corresponding to a trial 
material law and the measured total elastic energy over a range of prescribed 
boundary displacements. The question of material identifiability is thus 
reduced to the question of existence and uniqueness of controls. 

We propose a specific functional framework combining: bounded sets of 
Lipschitz-continuous displacements; bounded sets of Lipschitz-continuous 
energy densities; and a minimax cost function. We endow these spaces with 
uniform convergence topologies. Within this framework, we prove existence of 
optimal controls and show that the question of material identifiability 
hinges on the separating properties of the boundary data. In addition, we 
exploit the variational character of the framework to formulate provably 
convergent approximations based on restriction to dense sequences of trial 
energy densities.  

The proposed functional framework naturally suggests and supports 
approximation by {\sl maxout} neural networks, i.~e., neural networks of 
piecewise affine or polyaffine functions and a maximum, or {\sl join}, 
activation function. This connection is natural in view of the lattice 
structure of the space of energy densities, the natural representation of 
convex (resp.~polyconvex) functions as upper envelops of piecewise affine 
(resp.~polyaffine) functions, and the minimax structure of the identification 
problem. Indeed, we show that maxout neural networks have the requisite 
density property in the space of energy densities and therefore result in 
convergent approximations as the number of neurons increases to infinity. 

We conclude with the following remarks.

\begin{itemize}

\item[i)]\underline{Well-posedness and convergence.} The importance of 
    establishing the well-posedness, in the sense of existence of solutions 
    and convergence of approximations thereof, can hardly be 
    overemphasized. The ill-posedness of the cost function and numerical 
    schemes is often masked in finite-dimensional calculations built on 
    discrete models, e.~g., finite elements and neural networks, when 
    working at a fixed level of discretization without assessing the 
    convergence properties of the scheme. Under those conditions, any 
    arbitrary continuous cost function minimized over a bounded set of 
    trial energy densities will return a solution by the Weierstrass 
    extreme value theorem. However, there is no guarantee in general that 
    such numerical solutions converge in any reasonable sense to the 
    underlying material law as the number of degrees of freedom is 
    increased. 
     
\item[ii)]\underline{Neural networks representations.} Neural-network 
    representations are widely used in applications, but often the topology 
    and structure of the network is {\sl ad hoc} or tuned by trial and 
    error, without a clear sense of approximation and optimality. As 
    already noted, in the present framework the particular class of maxout 
    neural networks is suggested naturally by the variational framework of 
    the problem and the lattice structure of the space of energy densities, 
    and is found to result in convergent approximations. Conveniently, 
    maxout neural networks partake of said lattice structure and are 
    therefore 'flat', in the sense that all layerings, or 
    topologies, of the network are equivalent. This property eliminates the 
    need to optimize the network topology, which is a difficult task for 
    general neural network representations. 

\item[iii)]\underline{Numerical implementation.} The connection with neural 
    networks opens avenues for the efficient numerical implementation of 
    the material identification approach. The overriding advantage of 
    maxout neural networks is that they can be {\sl trained} efficiently 
    \cite{Goodfellow:2013, Strang:2019}. Conveniently, due to the 
    piecewise-constant form of the trial energy densities in the maxout 
    neural network representation, the identification problem reduces to a 
    linear programming form, which, upon spatial discretization, can be 
    solved by means of linear programming solvers. The interested reader is 
    referred to \cite{Conti:2024} for further details.  

\item[iv)]\underline{High-throughput experiments.} Traditional mechanical 
    testing is for the most part based on experimental designs that ensure 
    the direct inference of stresses and strains from measurements, e.~g., 
    by endeavoring to induce states of constant stress and strain, or 
    simple wave patterns, or by some other means. Because of the paucity of 
    material state information contained in these states, appropriate 
    coverage and characterization of material behavior requires an 
    extensive suite of tests to be performed. Recent work \cite{Jin:2022}, 
    suggests that high-throughput experiments can be designed by inducing 
    highly-inhomogeneous--possibly singular--deformations in the specimen, 
    thus achieving extensive coverage with a single test. The present 
    variational framework provides an avenue for interpreting and 
    optimizing the design of such high-throughput experiments. 
    Specifically, the optimal testing program is set forth by the smallest 
    set $M$ of prescribed boundary displacements that ensures 
    identifiability of the material law, i.~e., the smallest {\sl 
    separating} set $M$ in the sense of Definition~\ref{Lawj4j}. We note 
    that, when the approximating spaces $U_h$ are finite-dimensional, 
    separation requires consideration of a finite applied-displacement set 
    $M_h$, as illustrated by the examples in Section~\ref{Jz8Rgr}.  

\end{itemize}

\section*{Acknowledgements} 

This work was funded by the Deutsche Forschungsgemeinschaft (DFG, German 
Research Foundation) {\sl via} project 211504053 - SFB 1060; project 
441211072 - SPP 2256; and project 390685813 -  GZ 2047/1 - HCM.

\end{document}